\documentclass{article}

\usepackage{amsmath, amssymb, amsthm}
\usepackage{cite}
\usepackage{authblk}

\newtheorem{theorem}{Theorem}
\newtheorem{defn}[theorem]{Definition}
\newtheorem{lemma}[theorem]{Lemma}
\newtheorem{corl}[theorem]{Corollary}

\numberwithin{theorem}{section}
\numberwithin{equation}{section}

\newcommand{\dz}{\partial_z}
\newcommand{\disk}{\mathbb{D}}
\newcommand{\zbar}{\overline{z}}
\newcommand{\dzbar}{\partial_{\zbar}}
\newcommand{\suminf}{\sum_{n = 1}^{\infty}}
\newcommand{\sumN}{\sum_{n = 1}^{N}}
\newcommand{\eps}{\epsilon}
\newcommand{\kbar}{\overline{K}}
\newcommand{\C}{\mathbb{C}}

\title{Stretching and Rotation Sets of Quasiconformal Mappings}
\author{Tyler Bongers %
	\thanks{Email: \texttt{charlesb@math.msu.edu} }
}

\affil{Department of Mathematics, Michigan State University}

\date{\today}

\begin{document}
\maketitle

\begin{abstract}
Quasiconformal maps in the plane are orientation preserving homeomorphisms that satisfy certain distortion inequalities; infinitesimally, they map circles to ellipses of bounded eccentricity. Such maps have many useful geometric distortion properties, and yield a flexible and powerful generalization of conformal mappings. In this work, we study the singularities of these maps, in particular the sizes of the sets where a quasiconformal map can exhibit given stretching and rotation behavior. We improve results by Astala-Iwaniec-Prause-Saksman and Hitruhin to give examples of stretching and rotation sets with non-sigma-finite measure at the appropriate Hausdorff dimension. We also improve this to give examples with positive Riesz capacity at the critical homogeneity, as well as positivity for a broad class of gauged Hausdorff measures at that dimension.
\end{abstract}

\tableofcontents
%
%
\section{Introduction}
We say that a map $f \in W^{1, 2}_{\text{loc}}(\mathbb{C})$ is $K$-quasiconformal if it is an orientation preserving homeomorphism and satisfies the distortion inequality
$\max_{\alpha} |\partial_{\alpha} f| \le K \min_{\alpha} |\partial_{\alpha} f|$ almost everywhere, where $\partial_{\alpha}$ is a directional derivative. Geometrically, $f$ maps infinitesimal circles to infinitesimal ellipses; these can be viewed as perturbations of conformal maps, which are $1$-quasiconformal. Such maps are also realized as solutions to the Beltrami equation
$$\partial_{\overline{z}} f = \mu(z) \partial_z f$$
where the coefficient $\mu$ satisfies $\|\mu\|_{\infty} \le \frac{K - 1}{K + 1} < 1$. 

We are interested in geometric distortion properties of these maps. Given $z \in \mathbb{C}$, we say that $f$ stretches with exponent $\alpha$ and rotates with exponent $\gamma$ at $z$ if there exist scales $r_n \to 0$ with
$$\lim_{n \to \infty} \frac{\log |f(z + r_n) - f(z)|}{\log r_n} = \alpha \text{ and } \lim_{n \to \infty} \frac{\operatorname{arg}(f(z + r_n) - f(z))}{\log |f(z + r_n) - f(z)|} = \gamma.$$
Here, the argument is interpreted as the total angular change with respect to $f(z)$ along the image of the ray $[z + r_n, \infty)$; see section 2 or \cite{AIPS} for the full definitions. 

A classical theorem of Mori (see \cite{Mori}) states that every $K$-quasiconformal map is locally $1/K$-H\"older continuous, which implies that $1/K \le \alpha \le K$. In the more recent \cite{AIPS}, Astala, Iwaniec, Prause and Saksman improved this substantially to give the exact range of both stretching and rotation exponents which can be realized by a $K$-quasiconformal map $f$: if we let $B_K \subset \mathbb C$ be the open disk centered at $\frac 1 2 (K + \frac 1 K)$ with radius $\frac 1 2 (K - \frac 1 K)$, then $f$ can stretch like $\alpha$ and rotate like $\gamma$ if and only if $\alpha(1 + i \gamma) \in \overline{B_K}$. As a particular application, this gives the precise rotation behavior that a bilipschitz map can exhibit. Moreover, this work gave the precise multifractal spectrum $F_K(\alpha, \gamma)$ - that is, the maximal possible Hausdorff dimension of the simultaneous stretching and rotation set of such maps; the sharp result was the following theorem.
\begin{theorem}\label{thm:main_theorem1}
If $f : \mathbb{C} \to \C$ is a $K$-quasiconformal mapping with $K > 1$, and $\alpha(1 + i \gamma) \in B_K$, then the Hausdorff dimension of the stretching and rotation set $E_f$ of $f$ is bounded by
$$\dim_{\mathcal{H}} E_F \le F_K(\alpha, \gamma) := 1 + \alpha - \frac{K + 1}{K - 1} \sqrt{(1 - \alpha)^2 + \frac{4K \alpha^2 \gamma^2}{(K + 1)^2}}$$
and this result is sharp at the level of dimension.
\end{theorem}
The techniques used to prove this theorem mainly involved improved integrability estimates for complex powers of the derivatives of $f$. There is very substantial overlap with the techniques used in studying area distortion, and as such it is a natural conjecture that the Hausdorff measure at the appropriate dimension should be finite, in analogy with Theorem 1.2 in \cite{ACTUTV}. However, we will show that this is not the case.

In the direction of lower bounds, that paper gives constructions to attain all dimensions below the bound $F_K(\alpha, \gamma)$. Hitruhin improved this in \cite{Hitruhin} to give examples of quasiconformal maps whose stretching and rotation sets have positive and finite Hausdorff measure at the critical dimension. That paper used a Cantor set construction from \cite{UT08} to prove this; the work gives a construction of a quasiconformal map whose distortion of a family of disks is completely understood.

%
%
In this work, we improve the above results beyond finite measure, showing that the stretching and rotation set can actually have positive measure with respect to many gauged Hausdorff measures which are much smaller than the typical $\mathcal{H}^d$. Our main theorem is
\begin{theorem}\label{thm:main_theorem2}
Let $\Lambda$ be a gauge function of the form $\Lambda(r) = r^d h(r)$ where $h$ is a nonnegative, nondecreasing function satisfying the growth condition $h(r) / h(s) \ge C_{\eps} (r/s)^{\eps}$ for all $\epsilon > 0$ and $0 < r \le s$ sufficiently small. Select parameters $\alpha < 1$ and $\gamma$ such that $d > 0$ is the maximum allowed Hausdorff dimension of the corresponding stretching and rotation set. Then there is a $K$-quasiconformal mapping $f$ and a set $E$ with $\mathcal{H}^{\Lambda}(E) > 0$ such that $E$ is the stretching and rotation set for $f$.
\end{theorem}

We have a generalization to stretching exponents $\alpha > 1$ under an additional constraint on the gauge function $\Lambda$. Furthermore, as a corollary, there is an application to an interesting class of gauge functions:
\begin{corl}\label{corl:main_corl1}
There are positive measure stretching and rotation sets associated to the gauges $\Lambda(r) = r^d \left(\log \frac 1 r\right)^{-\beta}$ for every $\beta > 0$.
\end{corl}

As an interesting second corollary, we can extend this to positive Riesz capacity $\dot{C}_{\beta, p}$ for all parameter choices $(\beta, p)$ with homogeneity matching the dimension $d$. In this case, we will be able to relate capacity results to gauge functions; this is also connected to the work in \cite{ACTUTV}.

The paper is organized as follows. In Section 2, we give a brief recollection of some notions involving quasiconformal mappings, and a more precise definition of the rotation. In Section 3, we analyze the Hausdorff dimension zero case; our main results here will be a construction of a quasiconformal mapping that stretches on any given countable set, as well as a first construction of a map with $\mathcal{H}^d$ non-$\sigma$-finite stretching and rotation set, where $d = F_K(\alpha, \gamma)$. In Section 4, we will prove the main theorem and indicate applications to particular gauges and Riesz capacities.

\section{Prerequisites}
%
%
Following \cite{AIPS}, given a quasiconformal map $f$, we will say that it \emph{stretches like $\alpha$} at a point $z_0$ if there exists a sequence of scales $r_n$ decreasing to zero for which
$$\lim_{n \to \infty} \frac{\log |f(z_0 + r_n) - f(z_0)|}{\log r_n} = \alpha.$$
Rotation is similar, but a little more subtle. For a principal quasiconformal map $f$, that is a map whose domain and codomain are both $\mathbb{C}$ and $f(z) = z + O\left(\frac 1 z\right)$ as $|z| \to \infty$, we can select a branch of $\log f$. We can find a corresponding choice of argument, and using this we can understand
$\operatorname{arg} (f(z_0 + r) - f(z_0))$
as the total rotation around the point $f(z_0)$ of the image of the ray $[z_0 + r, \infty)$ under $f$. Using this interpretation, we will say that $f$ \emph{rotates like $\gamma$} at a point $z_0$ if
$$\lim_{n \to \infty} \frac{\operatorname{arg}(f(z_0 + r_n) - f(z_0))}{\log |f(z_0 + r_n) - f(z_0)|} = \gamma$$
for a sequence $r_n \to 0$. It is worth noting that the stretch and rotation at a point are not uniquely defined; it is possible that a quasiconformal map stretches like $\alpha$ and $\alpha'$ at a point with $\alpha \ne \alpha'$ (or rotates with two different behaviors); this is due to the dependence on the particular choice of sequence $r_n$.

Given a quasiconformal mapping $f$, we set $E_f(\alpha, \gamma)$ to be its simultaneous rotation-like-$\gamma$ and stretching-like-$\alpha$ set; when it is clear from context, this will be abbreviated as $E_f$. Finally, we have the multifractal spectrum
$$F_K(\alpha, \gamma) = \sup \left\{\dim_{\mathcal{H}}(E_f(\alpha, \gamma)): f \text{ is $K$-quasiconformal}\right\}$$
where this $F_K(\alpha, \gamma)$ is that of Theorem \ref{thm:main_theorem1}, as proved in \cite{AIPS}.
%
%

\section{Dimension Zero}

%
%

There are two complementary senses in which we will improve upon results with positive measure. The first is to give particular examples of stretching and rotation sets with very large measure, perhaps uncountable or having positive measure with respect to some gauged Hausdorff measure. The second is to give a broader class of examples of sets, in particular including that every countable set can appear as a stretching set. Before the constructions, we will start with a useful lemma that will allow us to simplify some of the subsequent computations involving stretching. Although it was not stated as a separate result, the computation here is more or less contained in \cite{Hitruhin}.

%
%

\begin{lemma}\label{lemma:generalApproximation}
Suppose that $z$ is a point with the following property: there is a sequence of balls $B_n = B(z_n, r_n)$ such that $z \in B_n$ for each $n$, $r_n \to 0$, and 
$$\frac{\log|f(B_n)|}{\log |B_n|} = \alpha + \epsilon_n$$
with error $\epsilon_n \to 0$. Then $f$ stretches like $\alpha$ at $z$.
\end{lemma}

The utility of this lemma is that we can transfer stretching information at a central point not only to points at difference $r$ away, but to \emph{all} nearby points. As an idea of an application, it is frequently possible to get stretching at exponent $\alpha$ on an entire Cantor set just by taking a quasiconformal map that stretches like $\alpha$ at each of the points used at successive scales to generate the Cantor set. 

\begin{proof}
Fix $n$. We can rotate using quasisymmetry. Fix a point $w \in \partial \disk(z_n, r_n)$ that is equidistant with $z$ and $z_n$ (e.g. an intersection point of the perpendicular bisector of $\overline{z z_n}$ with the boundary of the circle). Then 
\begin{align*}
\log |f(z + r_n) - f(z)| &= \log|f(z + r_n e^{i \theta}) - f(z)| + C_K \\
&= \log |f(w) - f(z)| + C_K \\
&= \log |f(w) - f(w + |z - w| e^{i \nu})| + C_K \\
&= \log |f(w) - f(z_n)| + C_K' \\
&= \log |f(z_n + |w - z_n|) - f(z_n)| + C_K'' \\
&= \log |f(z_n + r_n) - f(z_n)| + C_K''
\end{align*}
given appropriate choices of $\nu$ and $\theta$; the constants $C_{K}, C_K'$ and $C_K''$ are unimportant except in that they are bounded in terms of $K$ only. Dividing by $\log r_n$ and letting $n \to \infty$, we find that
\begin{align*}
\frac{\log|f(z + r_n) - f(z)|}{\log r_n} &= \frac{\log|f(z_n + r_n)| + C_K''}{\log r_n} \\
&= \frac{\frac 1 2 \log |f(B_n)| + C_K'''}{\frac 1 2 \log |B_n| - \frac 1 2 \log \pi} \\
&= \alpha + \epsilon_n + \frac{2 C_K'''}{\log|B_n|} + o(1)
\end{align*}
following a final application of quasisymmetry. The result follows.
\end{proof}

Note that we can replace the measures of the balls with their radii. We can actually extract a little more information: if $C$ is a fixed constant, and $z$ is a point for which $|z - z_n| \le Cr_n$, the same result holds. To see this, notice that there is a polygonal path connecting $z$ to $z_n$ where each segment has length $r_n$, and the number of segments is uniformly bounded by a constant only involving $C$. Repeating the double-rotation idea of the proof, we now lose a constant several times (but a uniformly bounded number), which does not impact the result. 

Moreover, the same result holds for rotations:
\begin{lemma}\label{lemma:generalApproximationRotation}
Suppose that $z$ is a point with the following property: there is a sequence of balls $B_n = B(z_n, r_n)$ such that $z \in B_n$ for each $n$, $r_n \to 0$, and 
$$\frac{\operatorname{arg}(f(z_n + r_n) - f(z))}{\log |f(z_n + r_n) - f(z_n)|} = \gamma + \epsilon_n$$
with error $\epsilon_n \to 0$. Then $f$ rotates like $\gamma$ at $z$.
\end{lemma}

\begin{proof}
We will only give a brief description of the technique of the proof, as it is rather similar to the previous one. Fix $n$, and consider the rays $[z_n, \infty)$ and $[z, \infty)$ parallel to the positive $x$-axis. By a rotation, which changes the cumulate argument by an $O(1)$ factor, we may assume that $z$ lies on the ray $[z_n, \infty)$. Now reusing the double rotation argument of the previous lemma, the denominators of the rotation are the same up to an $O(1)$ error, which is enough.
\end{proof}

Our first result will be a large dimension zero set which has the most extreme stretching and rotation allowed by the multifractal spectrum bounds of \cite{AIPS}. The construction will be a sort of Cantor set built from disks, within which we can explicitly keep track of the stretching and rotation.

%
%

\begin{theorem} \label{theorem:uncountableCantorSet}
For any pair $(\alpha, \gamma)$ for which $z |z|^{\alpha(1 + i \gamma) - 1}$ is $K$-quasiconformal, there is a $K$-quasiconformal map $f$ and an uncountable set $E_f$ for which $f$ stretches like $\alpha$ and rotates like $\gamma$ at every point in $E_f$.
\end{theorem}

\begin{proof} Start with $B_{0, 1} = \mathbb{D}$ and $f(z) = z$ on all of $\mathbb{C}$. Now assume that $B_{n, i}$ has been defined and has radius $r_n$, and that there are complex numbers $\beta_{n,i}, w_{n, i}$ for which $f(z) = \beta_{n,i} z + w_{n, i}$ in a neighborhood of $B_{n, i}$. Choose a number $\tilde{r}_n$ (which will be substantially smaller than $r_n$); take a concentric ball $A_{n, i}$ within $B_{n, i}$ of radius $\tilde{r}_n$, and place two disjoint balls $B_{n + 1, j}$ within $A_{n, i}$ each with radius $\frac 1 4 \tilde{r}_n$. We now modify the construction of $f$; without loss of generality, we may assume that $w_{n, i} = 0$ and $f(w_{n_i}) = 0$ - otherwise, pre- and post-compose with an appropriate translation (this only simplifies the notation). Now modify the definition of $f$ to become
$$
f(z) = \left\{\begin{array}{ll} 
	\beta_{n,i} z & \text{ near } B_{n, i}\text{ but in } B_{n, i}^c \\
	\beta_{n,i}  z \left|\frac{z}{r_n}\right|^{\alpha(1 + i \gamma) - 1} & z \in B_{n, i} \setminus A_{n, i} \\
	\beta_{n,i} \left(\frac{\tilde{r}_n}{r_n}\right)^{\alpha - 1} e^{i \theta} z & z \in A_{n, i}
\end{array}
\right.
$$
where $e^{i \theta}$ is chosen so that $f$ is continuous across $\partial A_{n, i}$, and
$$\beta_{n + 1, j} = \beta_{n, i} \left(\frac{\tilde{r}_n}{r_n}\right)^{\alpha - 1} e^{i \theta}.$$

Note that the original function $f$ is injective; on the other hand, the construction only carries out a local modification by stretching and rotating the ball $A_{n, i}$, and remains injective. Moreover, the limiting function of the construction is $K$-quasiconformal as long as the parameters $(\alpha, \gamma)$ are chosen to allow this. In particular, following \cite{Hitruhin}, we can choose $\alpha, \gamma$ to be any pair for which $F_K(\alpha, \gamma) = 0$. 

We just need to compute the change in argument induced by crossing the annulus between $B_{n, i}$ and $A_{n, i}$, find the corresponding stretching on scale $\tilde{r}_n$ with respect to the center point, and choose the sequence of radii carefully. Since
$$\left|\frac z r\right|^{\alpha(1 + i \gamma)} = \left|\frac{z}{r}\right|^{\alpha} e^{i \alpha \gamma \log |z/r|}$$
it is immediate that the change in argument across the annulus is $\alpha \gamma \log \frac{\tilde{r}_n}{r_n} + O(1)$. The numerator of the stretching with respect to the center point of $B_{n, i}$ on scale $\tilde{r}_n$ is 
$$\log \left|\beta_{n, i} \left(\frac{\tilde{r}_n}{r_n}\right)^{\alpha - 1} e^{i \theta} \tilde{r}_n \right| = \alpha \log \tilde{r}_n + \log |\beta_{n, i}| - (\alpha - 1) \log r_n.$$
As a consequence, we see that the overall stretching of $f$ with respect to the center point is
\begin{equation}\label{eq:dim_zero_stretch}
\frac{\log |f(\tilde{r}_n) - f(0)|}{\log \tilde{r}_n} = \alpha + \frac{\log |\beta_{n, i}|}{\log \tilde{r}_n} - (\alpha - 1) \frac{\log r_n}{\log \tilde{r}_n}
\end{equation}
while the overall rotation is
\begin{equation}\label{eq:dim_zero_rotation}
\frac{\operatorname{arg} \left(f(\tilde{r}_n) - f(0)\right)}{\log |f(\tilde{r}_n) - f(0)|} = \frac{\alpha \gamma \log \tilde{r}_n - \alpha \gamma \log r_n + O(1)}{\alpha \log \tilde{r}_n + \log |\beta_{n, i}| - (\alpha - 1) \log r_n}.
\end{equation}
Each $\beta_{n, i}$ has the same modulus $\beta_n$; the only potential difference is the exact rotation. We can easily compute this number from its definition, finding that
$$\beta_n = \left[\prod_{k = 0}^{n - 1} \frac{\tilde{r}_k}{r_k}\right]^{\alpha - 1}$$
As a consequence, we have that
\begin{align*}
\frac{\log |\beta_{n, i}|}{\log \tilde{r}_n} &= (\alpha - 1) \sum_{k = 0}^{n - 1} \frac{\log \tilde{r}_k - \log r_k}{\log \tilde{r}_n}
\end{align*}
Because $\tilde{r}_k < r_k < 1$, we can estimate all the terms roughly by the final term (provided that $\tilde{r}_k/r_k$ is decreasing, which it will be), finding
\begin{equation}\label{eq:dim_zero_error}
\left|\frac{\log |\beta_{n, i}|}{\log \tilde{r}_n}\right| \le 2(1 - \alpha) n \frac{\log \tilde{r}_{n - 1}}{\log \tilde{r}_n}
\end{equation}
We have already defined $r_{k + 1} = \frac 1 4 \tilde{r}_k$, and now we make the selection that
$$\tilde{r}_k = r_k^{k^2}$$
and the above error estimate (\ref{eq:dim_zero_error}) tends to zero. As an immediate consequence of this selection, we have that the stretching tends towards $\alpha$, while the rotation tends towards $\gamma$. This completes the proof.
\end{proof}

%
%

Now we will go in the other direction, finding that any countable set is a stretching set with the worst possible exponent. As a nice application, this shows that an interesting multifractal spectrum bound in the style of \cite{AIPS} is not possible for Minkowski dimension; see, e.g. Chapter 5 of \cite{Mattila} for constructions of countable sets with large Minkowski dimension. There are countable sets whose lower Minkowski dimension is arbitrarily close to $2$, and these can exhibit stretching of exponent $1/K$ at every point. The key idea here will be that sums of radial stretches are quasiconformal maps; in general, it is quite rare for a sum of quasiconformal maps to be quasiconformal (let alone injective). This idea will not work for rotations. 

Note, however, that this contrasts starkly with the possibilities in other dimensions. For example, a one dimensional set containing a smooth curve or a segment can never be a stretching set for an exponent other than $1$. To see this, consider the fact that if $f$ stretches with exponent $\alpha > 1$ at every point within a line segment, $f$ is flat at every point within that line. Explicitly, if $f$ is viewed as a single-variable function on this line, it is (classically) differentiable with derivative zero at every point, hence non-injective. Considering $f^{-1}$ shows why $f$ cannot stretch with exponent $\alpha < 1$.

\begin{theorem}\label{theorem:countableSets}
Given a countable set $\Lambda \subseteq \mathbb{D}$, there is a $K$-quasiconformal mapping $f$ such that for each $\lambda \in \Lambda$ there is a sequence of scales $r_m$ decreasing to zero for which
$$\lim_{m \to \infty} \frac{\log|f(\lambda + r_m) - f(\lambda)|}{\log r_m} = \frac 1 K.$$
\end{theorem}

Recall that $1/K$ is the most extreme possible exponent due to \cite{AIPS}.

\begin{proof}
Let us begin with the radial stretches 
$$f_{\lambda}(z) = (z - \lambda) |z - \lambda|^{\frac 1 K - 1} + \lambda$$
when $|z - \lambda| \le 1$, and the identity otherwise. These are $K$-quasiconformal mappings that satisfy a Beltrami equation with coefficient $\mu_{\lambda_n}$. Moreover, their derivatives $\dz f_{\lambda}$ have constant sign where they are defined. To wit,
$$\dz f_{\lambda} = \left(\frac 1 {2K} + \frac 1 2\right) |z - \lambda|^{\frac 1 K - 1}$$
within the disk $\lambda + \mathbb{D}$, and $1$ outside. It follows that if we sum such solutions, we can still have a solution to a Beltrami equation; in particular, assuming that derivatives and sums commute in this context, we have
\begin{align*}
\left|\dzbar \sum_{n = 1}^{\infty} \frac 1 {2^n} f_{\lambda_n}(z) \right| &= \left|\sum_{n = 1}^{\infty} \frac 1 {2^n} \dzbar f_{\lambda_n}(z)\right| \\
&= \left|\sum_{n = 1}^{\infty} \frac 1 {2^n} \mu_{\lambda_n}(z) \dz f_{\lambda_n}(z) \right| \\
&\le \sum_{n = 1}^{\infty} \frac 1 {2^n} \|\mu_{\lambda_n}\|_{\infty} |\dz f_{\lambda_n}(z)| \\
&= \frac{K - 1}{K + 1} \sum_{n = 1}^{\infty} \frac 1 {2^n} \dz f_{\lambda_n}(z) \\
&= \frac{K - 1}{K + 1} \dz \sum_{n = 1}^{\infty} \frac 1 {2^n} f_{\lambda_n}(z)
\end{align*}
Now given a countable set, we can therefore define a function 
\begin{equation}\label{eq:qc_sums}
f(z) = \sum_{n = 1}^{\infty} \frac 1 {2^n} f_{\lambda_n}(z).
\end{equation}
Modulo swapping the derivatives and the sum, we have shown that $f$ satisfies a Beltrami equation with coefficient bounded by $(K - 1) / (K + 1)$. This condition will follow very quickly from the dominated convergence theorem. Fix a test function $\varphi \in C^{\infty}_0(\mathbb{C})$ and integrate by parts:
\begin{align*}
\int f \partial_x \varphi &= \int \lim_{n \to \infty} \sumN \frac 1 {2^n} f_{\lambda_n} \partial_x \varphi \\
&= \lim_{n \to \infty} \sumN \int \frac{1}{2^n} f_{\lambda n} \partial x \varphi
\end{align*}
where we have used the fact that $|f(z)| \le \sum_n \frac{1}{2^n} |f_{\lambda_n}(z)| \le \sum_n \frac{1}{2^n} (|\lambda| + 1 + |z|) \le 2 + |z|$ from the estimate $|f_{\lambda}(z)| \le |z - \lambda|^{1/K} + |\lambda| \le 2$ within the disk $\lambda + \mathbb{D}$, and $|z|$ otherwise. Thus $f$ is bounded on the support of $\varphi$, and the above follows. Now integrate by parts in each summand to get
\begin{align*}
\int f \partial_x \varphi &= - \lim_{n \to \infty} \int \sumN \frac{1}{2^n} \partial_x f_{\lambda_n} \varphi
\end{align*}
Now $\varphi$ is bounded on its support, and $|\partial_x f_{\lambda_n}| \lesssim_K |z - \lambda_n|^{1/K - 1}$ is locally integrable (as $1/K - 1 > -1$), and summing in $n$ does not change this. Taking $\sum_{n = 1}^{\infty} \frac 1 {2^n} |z - \lambda_n|^{1/K - 1} |\varphi|$ as our dominating function, we again interchange the limits and find that
$$\int f \partial_x \varphi = - \int \left(\suminf \partial_x f_{\lambda_n}\right) \varphi$$
as desired. Now we have that $f$ has a weak derivative, which is a convergent sum of locally $L^2$ integrable functions. The same holds for $\partial_y$, and hence both $\dz$ and $\dzbar$. Now it follows immediately that $f \in W^{1, 2}_{\text{loc}}(\mathbb{C})$ and satisfies a Beltrami equation; thus, the measurable Riemann mapping theorem (see, for example, Theorem 5.3.2 of \cite{AIM}) gives us the following lemma:

%
%

\begin{lemma}\label{theorem:quasiconformalitySums}
Given a countable set $\{\lambda_n\}_{n = 1}^{\infty} \subseteq \mathbb{D}$, the function $f$ defined in (\ref{eq:qc_sums}) is $K$-quasiconformal.
\end{lemma}

We now claim that this function $f$ has the correct stretching behavior at each point in $\Lambda$. Fix $\lambda_n \in \Lambda$; we can assume that $\lambda_n = 0$. Morally, we proceed as follows: there are contributions to the stretching from terms on two scales, the nearby and the far away. We can arrange it so that nearby points $\lambda_m$ only have very large indices, so that the exponentially decaying weights will render this negligible; on the other hand, far away points have the advantage of the smoothness of the radial stretches.

Let us make this precise. We will show that
\begin{equation}\label{eq:stretch_error}
|f(r) - f(0)| = c r^{1/K} + o(r^{1/K})
\end{equation}
with a non-zero constant $c$, from which the theorem will follow. First of all, it is clear that the term $n = m$ contributes exactly $\frac{1}{2^m} r^{1/K}$; we will estimate away the remaining terms. To this end, we have for terms with $m \ne n$ that the difference is
\begin{align*}
\sum_{m \ne n} \frac{1}{2^m} (r - \lambda_m) |r - \lambda_m|^{\frac 1 K - 1} - \frac{1}{2^m} (-\lambda_m) |-\lambda_m|^{\frac 1 K - 1}
\end{align*}
After factoring a term $-\lambda_m |-\lambda_m|^{\frac 1 K - 1}$ from each summand and applying the triangle inequality, we need to estimate
$$\sum_{m \ne n} \frac{1}{2^m} |\lambda_m|^{1/K} \left|\left(1 - \frac r {\lambda_m}\right) \left|1 - \frac{r}{\lambda_m}\right|^{1/K - 1} - 1\right|.$$
%
%
To deal with the term within the absolute value, we need a simple estimate of a particular function:
\begin{lemma}\label{theorem:stretchingEstimate}
If $K > 1$,
$$\left|(1 + z)|1 + z|^{1/K - 1} - 1\right| \le C_0 \min\left\{|z|, |z|^{1/K}\right\}.$$
for a constant $C_0$ depending only on $K$.
\end{lemma}

\begin{proof}
For large values of $|z|$, the triangle inequality implies that this is controlled by a constant multiple of $|z|^{1/K}$, which is smaller (up to a constant) than $|z|$. So let us assume that $|z|$ is small, e.g. $|z| \le \frac 1 2$. Write $|1 + z| = 1 + y$ with $y$ real and $|y| \le |z|$. 

If $y = 0$, $|1 + z| = 1$ and
$$(1 + z)|1 + z|^{1/K - 1} - 1 = z.$$
Otherwise, select $\lambda$ so that $\lambda y = z$; then Taylor expansion gives
\begin{align*}
(1 + z) |1 + z|^{1/K - 1} - 1 &= (1 + \lambda y) (1 + y)^{1/K - 1} - 1 \\
&= 1 + \left(\lambda + \frac 1 K - 1\right) y + O(y^2) - 1 \\
&= \left(\lambda + \frac 1 K - 1\right) y + O(y^2) \\
&= z + \left(\frac 1 K - 1\right) y + O(y^2) \\
&= z + O(|z|) + O(|z|^2)
\end{align*}
from which the lemma follows. 
\end{proof}

Now we are ready to make the division into two scales. The cutoff point is to separate in the following way: Since the sequence is fixed, we can choose $r$ small enough that
$$\frac{r}{|\lambda_m|} \ge \left(\frac{1}{2^{n + 1} C_0}\right)^{\frac{1}{1 - 1/K}} \implies m \ge n + a + 10$$
where $a$ is chosen so that $2^a > C_0$; $C_0$ here is the constant of Lemma \ref{theorem:stretchingEstimate}.
That is, when $|\lambda_m|$ is smaller than a very large constant multiple of $r$, the index must be very large.

The far scale is for terms when $(r/|\lambda_m|)^{1 - 1/K} < 1/2^{n + 1} C_0$. In this case we have the lemma's linear estimate available, and the sum over these indices $m$ is at most
$$C \sum_{m \text{ far}} \frac{1}{2^m} |\lambda_m|^{1/K} \frac{r}{|\lambda_m|} = C_0r^{1/K} \sum_{m \text{ far}} \frac{1}{2^m} \left(\frac{r}{|\lambda_m|}\right)^{1 - 1/K} < \frac{r^{1/K}}{2^{n + 1}}$$
which is enough. Note that we have no control over the index $m$ here.

Next is the nearby scale where we have the opposite inequality; now $m$ must be large but we have worse control on the summands. Using the non-linear estimate from the lemma, we find that the contribution is at most
$$C_0 \sum_{m \text{ near}} \frac{1}{2^m} |\lambda_m|^{1/K} \left(\frac{r}{|\lambda_m|}\right)^{1/K} = C_0\sum_{m \text{ near}} \frac{1}{2^m} r^{1/K} \le \frac{C_0}{2^{n + a + 9}} r^{1/K} < \frac{r^{1/K}}{2^{n + 9}}$$
having used the fact that $\sum_{m \ge N} \frac{1}{2^m} = \frac{1}{2^{N+1}}$.

Combining these two estimates, the contribution from all indices $m \ne n$ is of the order $r^{1/K}$ with constant significantly less than $2^{-n}$. This proves (\ref{eq:stretch_error}) and is the desired result.
\end{proof}

%
%

%
%
\section{Dimension Greater than Zero}

%
%

To prepare for the main result, we will define a particular class of gauge functions. These will be gauges which lead to minor perturbations of the pure Hausdorff meaures, without changing the dimension. The perturbations should be chosen to tend to zero slowly enough to guarantee this, and will contain some sort of embedded convexity condition.

\begin{defn}\label{defn:admissibleGauge} We will say that a gauge function $\Lambda(r) = r^d h(r)$ is \textbf{admissible} if $h(r)$ is continuous, nonnegative, non-decreasing on $[0, \infty)$, and satisfies the following decay condition at the origin: For every $\epsilon > 0$, there exists a constant $C_{\eps}$ such that for any $0 < r \le s \le 1$,
$$\frac{h(r)}{h(s)} \ge C_{\eps} \left(\frac r s\right)^{\eps}.$$
\end{defn}

It will be proven later that functions of the form $(\log(1/r))^{-\beta}$ for $\beta > 0$ are admissible, giving a rich class of examples. We now come to the first theorem of the section.
%
%
\begin{theorem}\label{theorem:gaugedStretches}
Let $\Lambda$ be an admissible gauge function. Fix $K$ and $\alpha \in (1/K, 1)$, setting $d = F_K(\alpha, 0)$. Then there is a set $E$ with positive gauged Hausdorff measure $\mathcal{H}^{\Lambda}(E)$ and a $K$-quasiconformal map $f$ so that $f$ stretches like $\alpha$ at every point in $E$.
\end{theorem}

\begin{proof}
%
%
The main construction of the proof is taken from \cite{UT08}, although our choice of parameters will be different. We retain the notation from that paper, and for the sake of self-containment give a brief description of the construction. At each stage of the construction, we will pack a disk completely with disjoint disks, and then shrink these disks appropriately to build a set of the desired Hausdorff dimension. The quasiconformal map will stretch these shrunken disks appropriately.

\textbf{Step 1.} Select $m_{1, 1}$ disjoint disks $D^{i}_{1, 1}$ of radius $R_{1, 1}$ within the unit disk, followed by $m_{1, 2}$ disjoint disks (and disjoint from the previously constructed disks as well) $D^i_{1, 2}$ of radius $R_{1, 2}$, and so on. In this manner we pack the unit disk completely in area, leading to
$$\sum_{j = 1}^{\infty} m_{1, j} R_{1, j}^2 = 1.$$
It is important to note that we can assume that every $R_{1, j}$ is smaller than some fixed $\delta_1 > 0$, which is as small as we desire. Also for each radius associate a parameter $\sigma_{1, j} > 0$; these will be chosen later, but are all quite small. 

Next, we construct a first approximation of our quasiconformal map. Denote the center of the disk $D^{i}_{1, j}$ as $z^i_{1, j}$. Let $\psi^i_{1, j}(z) = z^i_{1, j} + (\sigma_{1, j})^K R_{1, j} z$, and define disks
\begin{align*}
D^i_j &= D(z^i_{1, j}, R_{1, j}) = \frac{1}{(\sigma_{1, j})^K} \psi^i_{1, j} (\mathbb{D}) \\
(D^i_j)' &= D(z^i_{1, j}, (\sigma_{1,j})^K R_{1, j}) = \psi^i_{1, j}(\mathbb{D})
\end{align*}
Then our first approximation is
$$\varphi_1(z) = \left\{\begin{array}{ll}
(\sigma_{1, j})^{1 - K} (z - z^i_{1, j}) + z^i_{1, j}, \quad & z \in (D^i_j)' \\
\left|\frac{z - z^i_{1, j}}{R_{1,j}}\right|^{\frac 1 K - 1} (z - z^i_{1, j}) + z^i_{1, j}, \quad & z \in D^i_j \setminus (D^i_j)' \\
z, \quad & z \notin \bigcup D^i_j.
\end{array}\right.$$
This is $K$-quasiconformal, being a modification of a radial stretch, and is conformal except for the annular regions between small disks $(D^i_j)'$ and their dilates $D^i_j$. In particular, it is important to note that $\varphi_1$ maps the disks of radius $(\sigma_{1, j})^K R_{1, j}$ onto other disks of radius $\sigma_{1, j} R_{1, j}$.

\textbf{Step 2.} We repeat the idea of the construction from the previous step. Choose $m_{2, 1}$ disjoint disks $D^i_{2, 1}$ with centers $z^i_{2, 1}$ of radius $R_{2, 1}$, and so on; again these will be subject to the constraint
$$\sum_{j = 1}^{\infty} m_{2, j} R_{2, j}^2 = 1.$$
Again, we can choose $R_{2, j}$ to be bounded by some $\delta_2 > 0$, but as small as needed; this is the difference from step 1, as we may wish to have $\delta_2 < \delta_1$. Next, we choose $\sigma_{2, j} > 0$.

As before, we follow this with an approximation of the quasiconformal map. Set $\psi^n_{2, k}(z) = z^n_{2, k} + (\sigma_{2, k})^K R_{2, k} z$, a radius $r_{\{2, k\},\{1,j\}} = R_{2, k} \sigma_{1, j} R_{1, j}$ and define disks
\begin{align*}
D^{i, n}_{j, k} &= D(z^{i, n}_{j, k}, r_{\{2, k\},\{1,j\}}) = \varphi_1 \left(\frac{1}{(\sigma_{2, k})^K} \psi^i_{1, j} \circ \psi^n_{2, k}(\mathbb{D})\right) \\
(D^{i, n}_{j, k})' &= D(z^{i, n}_{j, k}, (\sigma_{2, k})^K r_{\{2, k\},\{1,j\}}) = \varphi_1\left(\psi^i_{1, j} \circ \psi^n_{2, k} (\mathbb{D})\right)
\end{align*}
Now we define
$$g_2(z) = \left\{\begin{array}{ll}
(\sigma_{2, k})^{1 - K} (z - z^{i,n}_{j,k}) + z^{i,n}_{j,k}, \quad & z \in (D^{i,n}_{j,k})' \\
\left|\frac{z - z^{i,n}_{j,k}}{r_{\{2, k\},\{1,j\}}}\right|^{\frac 1 K - 1} (z - z^{i,n}_{j,k}) + z^{i,n}_{j,k}, \quad & z \in D^{i,n}_{j,k} \setminus (D^{i,n}_{j,k})' \\
z, \quad & \text{otherwise}.
\end{array}\right.$$
Finally, our second approximation is $\varphi_2 = g_2 \circ \varphi_1$. As before, this is a $K$-quasiconformal map equal to the identity outside the unit disk; the most important thing to note is that this map behaves essentially as a radial stretch, sending certain disks of radius $(\sigma_{1, j} \sigma_{2, k})^K R_{1, j} R_{2, k}$ to certain other disks of radius $(\sigma_{1, j} \sigma_{2, k}) R_{1, j} R_{2, k}$.
%
%

\textbf{Induction step.} Assuming that $N - 1$ steps of the construction have been fulfilled, we repeat the process, getting disks $D^i_{N, j}$ with centers $z^q_{N, p}$, radii $R_{N, p}$ and satisfying
$$\sum_{j = 1}^{\infty} m_{N, j} R_{N, j}^2 = 1.$$
As before, we have a constraint $R_{N, j} < \delta_N$ and parameters $\sigma_{N, j} > 0$.

We proceed with the next approximation of the quasiconformal map. Define radii
$$r_{\{N, p\}, \{N - 1, h\}, \dots, \{1, j\}} = R_{N, p} \sigma_{N - 1, h} r_{\{N - 1, h\}, \dots, \{1, j\}}$$
and maps $\psi^q_{N, p}(z) = z^q_{N, p} + (\sigma_{N, p})^K R_{N, p} z$. For multiindices $I = (i_1, ..., i_N)$ and $J = (j_1, ..., j_N)$, we define disks
\begin{align*}
D^I_J &= D(z^I_J, r_{\{N, p\}, \dots, \{1, j\}}) = \varphi_{N - 1} \left(\frac{1}{(\sigma_{N, p})^K} \psi^{i_1}_{1, j_1} \circ \cdots \circ \psi^{i_N}_{N, j_N} (\mathbb{D})\right) \\
(D^I_J)' &= D(z^I_J, (\sigma_{N, p})^K r_{\{N, p\}, \dots, \{1, j\}}) = \varphi_{N - 1} \left(\psi^{i_1}_{1, j_1} \circ \cdots \circ \psi^{i_N}_{N, j_N} (\mathbb{D})\right)  
\end{align*}
As usual, we set
$$g_N(z) = \left\{\begin{array}{ll}
(\sigma_{N, p})^{1 - K} (z - z^I_J) + z^I_J, \quad & z \in (D^I_J)' \\
\left|\frac{z - z^I_J}{r_{\{N, p\}, \dots, \{1, j\}}}\right|^{\frac 1 K - 1} (z - z^I_J) + z^I_J, \quad & z \in D^I_J \setminus (D^I_J)' \\
z, \quad & \text{otherwise}.
\end{array}\right.$$

This map is $K$-quasiconformal, conformal outside of the union of all the annuli and preserves the disks $D^I_J$. We finally set $\varphi_N = g_N \circ \varphi_{N - 1}$, noting that this is the identity outside the unit disk and maps disks of radius $(\sigma_{1, j_1} \cdots \sigma_{N, j_N})^K R_{1, j_1} \cdots R_{N, j_N}$ to disks of radius $(\sigma_{1, j_1} \cdots \sigma_{N, j_N}) R_{1, j_1} \cdots R_{N, j_N}$. 

We now take the limits resulting from this construction. As $\varphi_N$ is a $K$-quasiconformal map which is the identity outside of $\mathbb{D}$, compactness of quasiconformal maps allows us to select a $K$-quasiconformal limit
$$f = \lim_{n \to \infty} \varphi_N$$
with convergence in the Sobolev space $W^{1, 2}_{\text{loc}}$. 

To recap, the result of the above construction is a Cantor type set $E$ whose building blocks at generation $N$ are disks with radius 
$$s_{j_1...j_N} = \left((\sigma_{1, j_1})^K R_{1, j_1}\right)\dots\left((\sigma_{N, j_N})^K R_{N, j_N}\right)$$
which are mapped to disks of radius
$$t_{j_1...j_N} = \left(\sigma_{1, j_1} R_{1, j_1}\right)\dots\left(\sigma_{N, j_N} R_{N, j_N}\right)$$
where we can choose $\sigma_{i, j_i}$ more or less freely, subject to the constraint that they are all small.

%
%
Now we will select our parameters $\sigma_{k, j_k}$. We will choose them subject to the governing equation
\begin{align}
R_{1, j_1}^2 \cdots R_{N, j_N}^2 &= (R_{1, j_1} \cdots R_{N, j_N})^d (\sigma_{1, j_1} \cdots \sigma_{N, j_N})^{Kd} \nonumber \\
&\quad\quad \cdot h(R_{1, j_1} \cdots R_{N, j_N} \sigma_{1, j_1}^K \cdots \sigma_{N, j_N}^K). \label{eq:governing}
\end{align}
If we write $\sigma_{k, j_k} = R_{k, j_k}^{\frac{2 - d}{Kd}} \eta_{k, j_k}$, the condition is equivalent to
\begin{equation}\label{eq:governing_simplified}
1 = \eta_{1, j_1}^{Kd} \cdots \eta_{N, j_N}^{Kd} h\left(R_{1, j_1}^{2/d} \cdots R_{N, j_N}^{2/d} \eta_{1, j_1}^K \cdots \eta_{N, j_N}^K\right).
\end{equation}

To see the relevance of the governing equation, note that if we sum over all the building blocks of our construction at level $N$, our choice of parameters gives us
\begin{align*}
\sum_{j_1, ..., j_N} m_{1, j_1} \cdots m_{n, j_n} s_{j_1, \dots, j_n}^d h(s_{j_1, \dots, j_n}) &= \sum_{j_1, \dots, j_N} (R_{1, j_1} \cdots R_{N, j_N})^2 = 1
\end{align*}
This is suggestive of the desired result, namely that the constructed set has positive measure in the gauge $r^d h(r)$.

We now have three questions left to address: whether we can actually select our parameters $\sigma$ in this manner, whether the Cantor set will exhibit the correct stretching, and whether the set has positive measure with respect to $\mathcal{H}^{\Lambda}$. 

%
%
First, we consider the satisfiability of the governing equation for $\sigma_{k, j_k}$; the selection is made inductively. Looking at the second form of our governing equation, and recalling that $h$ is continuous, it is immediately clear that we can select $\eta_{N, j_N}$ to satisfy the equation - the right hand side tends to zero as $\eta_{N, j_N}$ does, and to infinity as $\eta_{N, j_N}$ does. The only concern is that $\eta_{N, j_N}$ might be so large as to defeat our requirement that $\sigma_{N, j_N}$ is small. First, notice that $R_{N, j_N} \sigma_{N, j_N}^K < 1$; if it were not, then we would have
\begin{align*}
R_{1, j_1}^2 \cdots R_{N, j_N}^2 &= (R_{1, j_1} \cdots R_{N - 1, j_{N - 1}})^d (\sigma_{1, j_1} \cdots \sigma_{N - 1, j_{N - 1}})^{Kd} \\
&\quad \quad \cdot h\left(R_{1, j_1} \cdots R_{N, j_N} \sigma_{1, j_1}^K \cdots \sigma_{N, j_N}^K\right) (R_{N, j_N} \sigma_{N, j_N}^K)^d \\
&\ge (R_{1, j_1} \cdots R_{N - 1, j_{N - 1}})^d (\sigma_{1, j_1} \cdots \sigma_{N - 1, j_{N - 1}})^{Kd} \\
&\quad \quad \cdot h\left(R_{1, j_1} \cdots R_{N-1, j_{N-1}} \sigma_{1, j_1}^K \cdots \sigma_{N-1, j_{N-1}}^K\right) \\
&= R_{1, j_1}^2 \cdots R_{N-1, j_{N-1}}^2
\end{align*}
contradicting the fact that each $R_{k, j_k}$ is much smaller than $1$. 

The above computation also suggests how to bound each $\sigma_{N, j_N}$, by playing the governing equation off itself at different generations. In this manner, essentially just rearranging the above, we find that
$$R_{N, j_N}^2 = R_{N, j_N}^d \sigma_{N, j_N}^{Kd} \frac{h\left(R_{1, j_1} \cdots R_{N, j_N} \sigma_{1, j_1}^K \cdots \sigma_{N, j_N}^K\right)}{h\left(R_{1, j_1} \cdots R_{N-1, j_{N-1}} \sigma_{1, j_1}^K \cdots\sigma_{N - 1, j_{N - 1}^K}\right)}$$
Rearranging for $\sigma_{N, j_N}$ and applying our growth condition with exponent $\epsilon$, we find that
$$\sigma_{N, j_N}^{Kd} \le R_{N, j_N}^{2 - d} \left(\frac{1}{R_{N, j_N} \sigma_{N, j_N}}\right)^{\eps} \frac{1}{C_{\eps}}.$$
Consequently,
$$\sigma_{N, j_N} \le \frac{1}{C_{\eps}^{1/K(d + \eps)}} R_{N, j_N}^{\frac{2 - d - \eps}{Kd}}.$$
As long as $\eps$ is chosen small enough that $2 - d - \eps > 0$, we may choose all $\delta_N$ small enough to result in $\sigma_{N, j_N} < 1/100$ as desired.

%
%
Next, we proceed to the stretching. Following the general approximation lemma \ref{lemma:generalApproximation}, it is sufficient to show that
$$\frac{\log t_{j_1,\dots,j_N}}{\log s_{j_1,\dots,j_N}} \to \alpha$$
as $N \to \infty$. In this direction, observe that
\begin{align*}
\frac{\log t_{j_1,\dots,j_N}}{\log s_{j_1,\dots,j_N}} &= \frac{\sum_{i = 1}^N \log R_{i, j_i} + \sum_{i = 1}^N \log \sigma_{i, j_i}}{\sum_{i = 1}^N \log R_{i, j_i} + K \sum_{i = 1}^N \log \sigma_{i, j_i}} \\
&= \frac{\left(1 + \frac{2 - d}{Kd}\right)\sum_{i = 1}^N \log R_{i, j_i} + \sum_{i = 1}^N \log \eta_{i, j_i}}{\left(1 + K \frac{2 - d}{Kd}\right) \sum_{i = 1}^N \log R_{i, j_i} + \sum_{i = 1}^N \log \eta_{i, j_i}}.
\end{align*}
Now provided that the perturbation terms are negligible with comparison to the radii terms, the stretching result follows. Indeed, in that case the quotient tends towards
$$\frac{1 + \frac{2 - d}{Kd}}{1 + K \frac{2 - d}{Kd}} = \frac{2 + (K - 1)d}{2K} = \alpha.$$
Thus, we need to prove that
$$S_N := \frac{\sum_{i = 1}^N \log \eta_{i, j_i}}{\sum_{i = 1}^N \log R_{i, j_i}}$$
tends to zero as $N$ grows.

To get this result, first notice that $S_N$ is negative: the product of all $\eta_{i, j_i}$ is greater than $1$ (as $h$ is small), while each $R_{i, j_i}$ is less than $1$; see (\ref{eq:governing_simplified}). From this, it follows that
\begin{align*}
0 \ge Kd S_N &= \frac{Kd \sum_{i = 1}^N \log \eta_{i, j_i}}{\sum_{i = 1}^N \log R_{i, j_i}} \\
&= \frac{- \log h\left(R_{1, j_1}^{2/d} \cdots R_{N, j_N}^{2/d} \eta_{1, j_1}^K \cdots \eta_{N, j_n}^K\right)}{\sum_{i = 1}^N \log R_{i, j_i}} \\
&\ge \frac{- \log \left(C_{\epsilon} R_{1, j_1}^{2\epsilon/d} \cdots R_{N, j_N}^{2\eps/d} \eta_{1, j_1}^{K \eps} \cdots \eta_{N, j_N}^{K \eps}\right)}{\sum_{k = 1}^N \log R_{i, j_i}} \\
&= \frac{-\log C_{\eps}}{\sum_{i = 1}^N \log R_{i, j_i}} - \frac{2\eps}{d} - K \eps S_N
\end{align*} 
where in the inequality we have used that $h(r) \ge C_{\epsilon} r^{\epsilon}$ provided that $r$ is sufficiently small, e.g. that $N$ is sufficiently large; this is the admisibility condition (\ref{defn:admissibleGauge}) applied with $s = 1$. To be precise, we require that $N$ is large enough that $R_{1, j_1}^{2\eps/d} \cdots \eta_{N, j_n}^{K\eps} < 1$. Now rearranging the result, we get
$$0 \ge S_N \ge \left(\frac{1}{K(d + \eps)}\right) \left(-\frac{\log C_{\eps}}{\sum_{i = 1}^N \log R_{i, j_i}} - \frac{2\eps}{d}\right)$$
It follows that we have
$$|S_N| \le \frac{\log C_{\epsilon}} {N \log 2} + O(\eps) = O(\eps)$$
provided that $N$ is chosen large enough given $\eps$. Taking $\eps$ to zero gives the required stretching.

%
%
Now all that remains is to show positivity of the measure of the Cantor set. Our starting point is an estimate analogous to equation (3.17) in \cite{UT08}; if $D$ is a building block at generation $N - 1$,
\begin{align}
\sum_{B_n \text{ children of } D} &r(B_n)^{d} h(r(B_n)) = \sum_{j_N} m_{N, j_N} \Lambda\left(R_{1, j_1} \cdots R_{N, j_N} \sigma_{1, j_1}^K \cdots \sigma_{N, j_N}^K\right) \nonumber \\
&= \left[R_{1, j_1} \cdots R_{N - 1, j_{N - 1}} \sigma_{1, j_1}^K \cdots \sigma_{N - 1, j_{N - 1}}^K\right]^d  \nonumber \\
&\quad \cdot \sum_{j_N} m_{N, j_N} (R_{N, j_N} \sigma_{N, j_N}^K)^d h\left(R_{1, j_1} \cdots R_{N, j_N} \sigma_{1, j_1}^K \cdots \sigma_{N, j_N}^K\right) \nonumber \\
&= \sum_{j_N} m_{N, j_N} R_{1, j_1}^2 \cdots R_{N, j_N}^2 \nonumber \\
&= R_{1, j_1}^2 \cdots R_{N - 1, j_{N - 1}}^2 \nonumber \\
&= (R_{1, j_1} \cdots R_{N - 1, j_{N - 1}})^d (\sigma_{1, j_1} \cdots \sigma_{N - 1, j_{N - 1}})^{Kd} \nonumber \\
&\quad \cdot h\left(R_{1, j_1} \cdots R_{N - 1, j_{N - 1}} \sigma_{1, j_1}^K \cdots \sigma_{N-1, j_{N-1}}^K\right) \nonumber \\
&= \Lambda(r(D)) \label{eq:carleson_first}
\end{align}
where we have used the governing equation (\ref{eq:governing}) at generations $N$ and $N - 1$. As a consequence, we can iterate this result to find that if $\{B_n\}$ is a finite collection of building blocks all contained in $D$ (not necessarily of the same generation), and $B_{N,k}$ are the generation $N$ descendents of $B_n$,
$$\sum_{B_n} \Lambda(r(B_n)) = \sum_{B_{N,k}} \Lambda(r(B_{N, k})).$$
We now wish to prove a Carleson style packing condition, from which positivity of measure will follow. We will state this as a separate lemma, similar to Lemma 3.2 of \cite{UT08}.
%
%
\begin{lemma}\label{lemma:carlesonEstimate}
Let $B$ be an arbitrary disk and $B_n$ disjoint building blocks of $E$. There is an absolute constant $C_1$ independent of the family $\mathcal{C} = \{B_n\}$ such that
$$\sum_{\substack{B_n \in \mathcal{C} \\ B_n \subset B}} \Lambda(r(B_n)) \le C_1 \Lambda(r(B)).$$
\end{lemma}

Once the lemma has been proven, the positivity of the gauged Hausdorff measure follows immediately. So let us fix such a family $\mathcal{C}$; we may assume that $r(B) \le 1$, since the above computation (\ref{eq:carleson_first}) shows that the lemma holds when $B = \mathbb{D}$. Choose the integer $H$ such that all the $B_n$ are contained in some building block at generation $H - 1$, but not at generation $H$; then let $\{B^H_{k_p}\}_{p = 0}^m$ be the complete list of ancestors at generation $H$ of our family. Note that the lemma holds with $B = B^{H - 1}_{i_0}$ (by the same reasoning that it holds for $B = \mathbb{D}$) and so we will assume that
$$r(B) \le r(B_{i_0}^{H - 1}) = s_{j_1, ..., j_{H - 1}}.$$
For each of these generation $H$ disks, let $\widetilde{B^H_{k_p}}$ be the concentric dilate with radius 
$$r(\widetilde{B^H_{k_p}}) = \frac{s_{j_1,...,j_H}}{\sigma_{H, j_H}^K}.$$
Provided that the multiindices $I = (i_1,...,i_H)$ and $J = (j_1,...,j_H)$ are chosen appropriately, these disks are the disks $(D^I_J)'$ from the construction of the Cantor set; now as each $\sigma_{N, p}$ is small (e.g. less than $1/100$) and since $B$ meets each $B^H_{k_p}$, we find that
$$2 r(B) \ge \frac{99}{100} r(\widetilde{B^H_{k_p}}).$$
Moreover, we have the containment $\widetilde{B^H_{k_p}} \subseteq 4B$.
We now can compute:
\begin{align*}
\sum_{B_n \in \mathcal{C}} \Lambda(r(B_n)) &\le \sum_{p = 0}^m \Lambda(r(B^H_{k_p})) \\
&= \left[\sigma_{1, j_1}^K R_{1, j_1} \cdots \sigma_{H - 1, j_{H - 1}}^K R_{H - 1, j_{H - 1}}\right]^d \\
&\quad\cdot \sum_{p = 0}^m \left(\sigma_{H, j_{H_{k_p}}}^K R_{H, j_{H_{k_p}}}\right)^d h\left(R_{1, j_1} \cdots \sigma_{H, j_{H_{k_p}}}^K\right) \\
&= s_{j_1, \dots j_{H-1}}^d h(s_{j_1, \dots j_{H-1}}) \sum_{p = 0}^m R_{H, j_{H_{k_p}}}^2 \\
&= s_{j_1, \dots j_{H-1}}^d h(s_{j_1, \dots j_{H-1}}) \frac 1 {\pi} \sum_{p = 0}^m \operatorname{Area}(D_p)
\end{align*}
where we have defined $D_p = D(z^{k_p}_{H, j_{H_{k_p}}}, R_{H, j_{H_{k_p}}})$, recalling that these are disks chosen during the induction step of the Cantor set's construction, called $D^i_{N, j}$. The second to last equality follows from applications of the governing equation at generations $H$ and $H - 1$.

Now since 
$$r(\widetilde{B^H_{k_p}}) = \frac{s_{j_1, \dots, j_{H_{k_p}}}}{\sigma_{H, j_{H_{k_p}}}^K}$$ 
and 
$$r(D_p) = R_{H, j_{H_{k_p}}} = \frac{r(\widetilde{B^H_{k_p}})}{s_{j_1, ..., j_{H - 1} } }$$
it follows that
\begin{align*}
\sum_{B_n \in \mathcal{C}} \Lambda(r(B_n)) &\le s_{j_1, \dots j_{H-1}}^d h(s_{j_1, \dots j_{H-1}}) \left[\frac{r(4B)}{s_{j_1, \dots, j_{H - 1}}}\right]^2 \\
&\lesssim r(B)^d h(s_{j_1, ..., j_{H - 1}}) \left[\frac{r(B)}{s_{j_1, ..., j_{H - 1}}}\right]^{2 - d}
\end{align*}
Finally, recall the condition (\ref{defn:admissibleGauge}) that for any $0 < x < y \le 1$, we have
$$\frac{h(x)}{h(y)} \ge C \left(\frac x y\right)^{2 - d}$$
Applying this to the above, it follows that
$$\sum_{B_n \in \mathcal{C}} \Lambda(r(B_n)) \lesssim r(B)^d h(r(B)) = \Lambda(r(B))$$
as desired.
\end{proof}

We now move to the rotation results.
%
%

\begin{theorem}\label{theorem:gaugedRotations}
Let $\Lambda$ be an admissible gauge function. Fix $K$ and parameters $\alpha, \gamma$ so that $\alpha(1 + i \gamma) \in B_K$ and $\alpha < 1$, setting $d = F_K(\alpha, \gamma)$. Then there is a set $E$ with positive gauged Hausdorff measure $\mathcal{H}^{\Lambda}(E)$ and a $K$-quasiconformal map $f$ so that $f$ stretches like $\alpha$ and rotates like $\gamma$ at every point in $E$.
\end{theorem}

\begin{proof}
%
%
This proof will very closely follow Hitruhin's modifications in \cite{Hitruhin} to add rotation to the previous theorem. We select $\overline{K} < 1/\alpha$ and let $\overline{f}$ be the $\overline{K}$-quasiconformal map previously constructed; the corresponding Cantor set has positive $\mathcal{H}^{r^{\overline{d}} h(r)}$ measure, where 
$$\overline{d} = 1 + \alpha - \frac{\overline{K} + 1}{\overline{K} - 1} (1 - \alpha).$$
Now all we need to do is modify the construction of $\overline{\varphi}_n$ for each $n$ by replacing the old $\overline{g}_n$ by
$$g_n(z) = \left\{\begin{array}{ll}
(\sigma_{n, j_n})^{1 - \overline{K}} (z - z^I_J) e^{i \theta^I_J} + z^I_J, \quad & z \in (D^I_J)' \\
\left|\frac{z - z^I_J}{r(D^I_J)}\right|^{\frac{1}{\overline{K}} - 1 + i \alpha \gamma \frac{\overline{K} - 1}{\overline{K}(1 - \alpha)}} (z - z^I_J) + z^I_J, \quad & z \in D^I_J \setminus (D^I_J)' \\
z, \quad & \text{otherwise}.
\end{array}\right.$$
where the change in argument over the annulus $D^I_J \setminus (D^I_J)'$ is $\theta^I_J$, and makes the map continuous across the boundary crossings. Let $f$ denote the resulting map using $\varphi_n$ and $g_n$, rather than the old versions $\overline{\varphi}_n$ and $\overline{g}_n$.

%
%
Since the paper \cite{Hitruhin} has already shown that $\overline{d} = F_K(\alpha, \gamma)$ is the desired dimension, and the previous theorem improves this to the perturbed Hausdorff gauge function, all that remains is to check that the rotational behavior is correct. That is, we need to show that
$$\lim_{n \to \infty} \frac{\operatorname{arg}(f(z_0 + r_n) - f(z_0))}{\log |f(z_0 + r_n) - f(z_0)|} = \gamma$$
for a suitable choice of scales $r_n \to 0$, and $z_0$ in a large subset of the Cantor set. Following the argument in \cite{Hitruhin}, we end up with the result that the total rotation as we move from $\infty$ to a disk at scale $n$ is
$$\operatorname{arg}\left(f(z_0 + r_n) - f(z_0)\right) = \alpha \gamma \frac{\kbar - 1}{(1 - \alpha)} \sum_{k = 1}^{n - 1} \log \sigma_{k, j_k} + O(n)$$
Now we select our parameters $\sigma_{k, j_k}$ as before, but with $\overline{d}$ and $\kbar$ replacing $d$ and $K$ respectively. With our usual notation
$$\sigma_{k, j_k} = R_{k, j_k}^{\frac{1 - \alpha}{\alpha \kbar - 1}} \eta_{k, j_k}$$
we can compute that
\begin{align*}
&\frac{\operatorname{arg}(f(z_0 + r_n) - f(z_0))}{\log |f(z_0 + r_n) - f(z_0)|} \\
&= \frac{\alpha \gamma \frac{\kbar - 1}{1 - \alpha} \left[\frac{1 - \alpha}{\alpha \kbar - 1} \sum_{k = 1}^{n - 1} \log R_{k, j_k} + \sum_{k = 1}^{n - 1} \eta_{k, j_k}\right]}{\alpha \left(\kbar \frac{1 - \alpha}{\alpha \kbar - 1} \sum_{k = 1}^{n - 1} \log R_{k, j_k} + \sum_{k = 1}^{n - 1} \log R_{k, j_k}  + \sum_{k = 1}^{n - 1} \eta_{k, j_k}\right)}\\
 &\approx \frac{\alpha \gamma \frac{\kbar - 1}{\alpha \kbar - 1}}{\alpha (\kbar \frac{1 - \alpha}{\alpha \kbar - 1} + 1)} \\
&= \frac{\alpha \gamma (\kbar - 1)}{\alpha \kbar (1 - \alpha) + \alpha \kbar - 1} \\
&= \gamma
\end{align*}
as desired, where we have used the previous result that $\sum_{k = 1}^N \eta_{k, j_k}$ is negligible in comaprison to $\sum_{k = 1}^N R_{k, j_k}$. In particular, letting $n \to \infty$, the infinitesimal rotation is exactly $\gamma$.
\end{proof}

Now we would like to generalize this theorem to include stretching exponents greater than $1$; this can be done by considering the inverse function $f^{-1}$, which inverts the stretching exponent and changes the sign of the rotation exponent. However, without assuming additional constraints on the perturbation $h$, it does not seem (to the best of the author's knowledge) possible to identify a gauge function $h'$ for which 
$$\mathcal{H}^{r^{d'} h'(r)}\left(f(E)\right) > 0.$$
It turns out that the key obstacle is a lack of decay in $h$; taking Section 4 of \cite{UT08} as inspiration, we will impose the additional condition that for all $t > 0$, $h(t) \lesssim h(t^K)$. Powers of logarithms such as $(\log 1/r)^{-\beta}$ clearly satisfy this condition, so we still have a useful class of examples.

\begin{theorem}\label{theorem:inverseGauge}
Let $\Lambda(r) = r^{d} h(r)$ be an admisible gauge function, with the additional constraint that $h(r^K) \lesssim h(r)$ for all $r > 0$. Let $E$ and $f$ be the stretching and rotation set and quasiconformal map constructed in Theorem \ref{theorem:gaugedRotations}, with exponents $\alpha$ and $\gamma$. Then $f^{-1}$ stretches with exponent $1/\alpha$ and rotates with exponent $-\gamma$ at every point in $f(E)$ and $f(E)$ has positive measure with respect to the gauge function
$$\Lambda'(r) = r^{d'} h^{d'/Kd}(r).$$
\end{theorem}

Since we have the additional decay constraint on $h$, the proof of this is a minor modification of that of Theorem 4.2 of \cite{UT08}, again proceeding through a Carleson type estimate. Rather than repeat a sketch of the argument, we will compare our conditions on $h$ to those of Uriarte-Tuero. First of all, for technical reasons, it is important to use only a finite family of disks (as in \cite{UT08}) at each generation of the Cantor set (rather, what is important is that there is a minimum choice of $R_{n, j}$ at each generation $n$, so that the construction of the next scale takes place on stricly smaller scales). In particular, choosing a sequence $\epsilon_n \to 0$ very quickly and packing an $(1 - \epsilon_n)$ portion of the unit disk at each generation will only change the measure by a factor $\prod_{n = 1}^{\infty} (1 - \epsilon_n) \approx 1$, so the finiteness condition is not an obstacle.

Secondly, it is required in Uriarte-Tuero's construction that $h(t)$ is a (strictly) increasing function for which $t^{\alpha} / h(t) \to 0$ as $t \to 0$ for each $\alpha > 0$, that $h^{1/(2 - d)}(t) / t$ is decreasing in $t$, and the logarithmic-type condition $h(t) \lesssim h(t^K)$. The first and fourth conditions hold here, as does the second by the definition of admissibility. Furthermore, admissibility applied with exponent $\eps = 2 - d$ gives us that if $r < s$,
$$\frac{h^{\frac{1}{2 - d}}(r)/r}{h^{\frac{1}{2 - d}}(s)/s} \ge \frac s r C_{\eps} \left(\frac{r}{s}\right)^{\frac{\eps}{2 - d}} = C_{\eps}$$
Although this does not show that $h^{1/(2 - d)}(t) / t$ is decreasing, it is almost decreasing (and in fact, if $s$ is small enough we may assume that $C_{\eps} > 1$, in which case we do have a decreasing function); it turns out that this is enough for the proof of the theorem to go through with minor modifications of the constants. It is worth pointing out, however, that the logarithmic-type decay condition is independent of the admissibility condition in the sense that neither is strong enough to imply the other. 

%
%
Now to show the usefulness of the theorems, it would be nice to give an explicit and interesting gauge function. Fortunately, logarithmic perturbations of $r^d$ are admissible, so we have a variety of gauges for which the theorem holds.
\begin{corl}\label{corl:logPerturbations}
There are positive measure stretching and rotation sets associated to the gauges $\Lambda(r) = r^d \left(\log \frac 1 r\right)^{-\beta}$ for every $\beta > 0$.
\end{corl}

To be precise, this is not a well-defined gauge function for $r \ge 1$; we ought to cut it off at some point between $0$ and $1$ so it does not blow up. However, as we only really care about the behavior as $r$ tends to zero, this is a point we will ignore; we will assume that $s \le \frac 1 {100}$.

\begin{proof}
We will show that the gauge functions $\Lambda(r) = r^d \left(\log \frac 1 r\right)^{-\beta}$ are admissible for all $\beta > 0$; all that we need to prove is the growth condition. Fix $0 < r \le s$ with $s$ small, and $\eps > 0$. We need to show that there exists a constant $C_{\eps, \beta}$ for which
$$\frac{h(r)}{h(s)} \ge C_{\eps, \beta} \left(\frac r s\right)^{\eps}$$
or alternatively, that
$$\frac{s^{\eps} (\log(1/s))^{\beta}}{r^{\eps}(\log(1/r))^{\beta}}$$
is bounded below independent of $r$ and $s$. We may just as well consider the functions
$$g(r, s) = \frac{s^{\epsilon/\beta} \log s}{r^{\epsilon/\beta} \log r}.$$ 
on the triangular domain $\{(r, s) : 0 < r \le s \le \frac{1}{100}\}$. 

First, let us fix $s$; we minimize the function over $r$. The $r$-derivative is 
$$\frac{\partial g}{\partial r} = \frac{s^{\eps/\beta} r^{\eps/\beta - 1}}{(r^{\eps/\beta} \log r)^2} (-\log s) \left[\frac{\eps}{\beta} \log r + 1\right],$$
which changes sign from negative to positive at $r = e^{-\beta/\eps}$. We now split into two cases, depending on the size of $s$.

The first case is that $s \ge e^{-\beta / \eps}$, so that $g(r, s)$ is in fact minimized at $r = e^{-\beta/\eps}$. In this case, we have
$$g(r, s) \ge g(e^{-\beta/\eps}, s) = -\frac{\eps e}{\beta} s^{\eps / \beta} \log s.$$ 
If we again differentiate, but in $s$, we get
$$-\frac{\eps e }{\beta} s^{\eps/\beta - 1} \left[\frac{\eps}{\beta} \log s + 1\right]$$
which is negative due to the fact that $s \ge e^{-\beta / \eps}$. Hence, this quantity is minimized when $s = \frac{1}{100}$; this gives a lower bound of
$$g(r, s) \ge g\left(e^{-\beta/\eps}, \frac 1 {100}\right) \quad\quad \forall\, 0 < r \le s, s \ge e^{-\beta / \eps}.$$
This of course only depends on $\beta$ and $\epsilon$, which is good enough.

The second case is that $s < e^{-\beta / \eps}$. Here, we may compute the $s$-derivative, finding that
$$\frac{\partial g}{\partial s} = \frac{1}{r^{\eps / \beta} \log r} \left[\frac{\eps}{\beta} s^{\eps/\beta - 1} \log s + s^{\eps/\beta} + \right] = \frac{s^{\eps/\beta - 1}}{r^{\eps/\beta}} \frac{1}{\log r} \left[\frac{\eps}{\beta} \log s + 1 \right].$$
Since $r < 1$, this is positive; therefore, $g$ increases from its minimum value of $1$ at the bottom of the domain where $r = s$, and is again bounded below.
\end{proof}

%
%
We can also apply this technique to get positive results for Riesz capacities. Recall that for a set $E$, the $(\beta, p)-$Riesz capacity $\dot{C}_{\beta, p}$ is defined by
$$\dot{C}_{\beta, p}(E) = \inf \left\{\|g\|_p : g \ast I_{\beta} \ge \chi_E \right\}$$
where up to a normalization, $I_{\beta}(z) = |z|^{-(2 - \beta)}$ is the Riesz kernel; see, e.g. \cite{AH} for more details. There is also a dual characterization by Wolff's theorem that
$$\dot{C}_{\beta, p}(E) \simeq \sup \left\{\mu(E) : \text{supp}(\mu) \subseteq E, \dot{W}^{\mu}_{\beta, p}(z) \le 1 \, \forall z \in \mathbb{C}\right\}$$
where the homogeneous Wolff potential $\dot{W}^{\mu}_{\beta, p}$ is 
$$\dot{W}^{\mu}_{\beta, p}(z) = \int_0^{\infty} \left(\frac{\mu\big(B(z, r)\big)}{r^{2 - \beta p}}\right)^{p' - 1} \frac{dr}{r}.$$
Furthermore, it is important to note that the Riesz capacity is homogeneous of degree $2 - \beta p$, which will correspond with the Hausdorff dimension of the set under consideration. Our main result here is the following theorem:

%
%
\begin{theorem}\label{theorem:postiveRieszCapacity}
Fix any parameter $\tau = \alpha(1 + i \gamma) \in B_K \setminus \{1\}$, and a pair $(\beta, p)$ with $1 < p < \infty$ and $2 - \beta p = F_K(\alpha, \gamma)$. There is a $K$-quasiconformal map $f$ and a set $E$ such that $f$ stretches with exponent $\alpha$ and rotates with exponent $\gamma$ at every point in $E$, and $E$ has positive $(\beta, p)-$Riesz capacity.
\end{theorem}

In particular, this shows that there cannot be a theorem improving the results of \cite{AIPS} to the level of Riesz capacity zero for any choice of parameters with the correct homogeneity. This stands in sharp contrast with the results of \cite{ACTUTV}, in which Riesz capacities were used to give sharper results than gauge functions alone can give. In \cite{ACTUTV}, at the critical homogeneity, there was a range of parameters $(\beta, p)$ in which extremal examples could exist, beyond which there was a negative result showing the sharpness of Riesz capacities. However, in our case, all possible indices have associated examples; thus an analogue of their theorem is not possible.

\begin{proof}
This theorem is actually much easier to prove than the last one, as the Riesz capacities of these Cantor type sets have already been estimated in \cite{ACTUTV}. We will first make the construction for a fixed $(\beta, p)$, and then extend it in such a way that the set will have positive Riesz capacities for all parameter choices simultaneously. Let $E$ be a Cantor type set as constructed in Theorem \ref{theorem:gaugedStretches}; our choice of parameters will be
$$\sigma_{k, j_k} = R_{k, j_k}^{\frac{2 - d}{dK}} \left(\frac{k + 1}{k}\right)^{\delta}$$
with $\delta$ to be chosen soon. Following the techniques of the previous proofs, we can compute that the stretching exponent is $\alpha$ at every point of $E$, while the rotation exponent is $\gamma$ on a large subset of $E$. 

Now it remains to understand the Riesz capacity of this set. Per Lemma 8.1 of \cite{ACTUTV}, if $\nu$ is the naturally distributed measure on $E$, its Wolff potential is 
$$\dot{W}^{\nu}_{\beta, p} \simeq \sum_{n = 2}^{\infty} \frac{1}{n^{dK (p' - 1) \delta}}$$
at each $x \in E$. If we select, e.g. $$\delta = 1 + \frac{1}{dK(p' - 1)}$$ then this series is convergent, the Wolff potential is uniformly bounded, and therefore the set has positive $(\beta, p)-$Riesz capacity.

Now we need to extend this from a particular parameter choice to all simultaneously. Carry out the above construction, but localized to a disk of radius $1/2$. Fix a new choice $(\beta_1, p_1)$ with $p_1 > p$, and carry out the construction with this parameter choice (meaning, with the updated value of $\delta$) in a disjoint disk of radius $1/4$. Continue in this manner with $(\beta_2, p_2)$ with $p_2 > p_1$, and so on; this gives a set with positive Riesz capacity for a sequence $(\beta_n, p_n)$ with $2 - \beta_n p_n = d$ for every $n$. If $p_n \to \infty$, a comparison theorem (e.g. Theorem 5.5.1(b) of \cite{AH}) shows that $E$ has positive capacity for all parameter choices.
\end{proof}

%
%
It is worth remarking that this theorem actually follows from the previous one, with the correct choice of $h$ (at least for these Cantor type sets). If we choose the gauge to be $h(r) = (\log 1/r)^{-1}$ for small enough $r$, then the resulting Cantor set must be larger, in a sense, than one only with positive Riesz capacity. This follows from an estimate of the $\eta_{k, j_k}$. Recall the generating relationship (\ref{eq:governing_simplified}):
$$1 = \left(\eta_{1, j_1} \cdots \eta_{N, j_N}\right)^{Kd} h\left(R_{1, j_1}^{2/d} \cdots R_{N, j_N}^{2/d} \eta_{1, j_1}^K \cdots \eta_{N, j_N}^K\right).$$

We will show that the choice of $\eta_{k, j_k}$ to satisfy this equation with this choice of gauge is typically larger than $(1 + 1/k)^{\delta}$; in particular, that means that the Cantor set naturally associated to this gauge is significantly larger than that constructed for positive Riesz capacity. To this end, suppose that $\eta_{k, j_k} \le (1 + 1/k)^{\delta}$ for all $k$. Then we have
\begin{align*}
1 &= \left(\eta_{1, j_1} \cdots \eta_{N, j_N}\right)^{Kd} h\left(R_{1, j_1}^{2/d} \cdots R_{N, j_N}^{2/d} \eta_{1, j_1}^K \cdots \eta_{N, j_N}^K\right) \\
&= (N + 1)^{K d \delta} h\left(R_{1, j_1}^{2/d} \cdots R_{N, j_N}^{2/d} (N + 1)^{K\delta}\right) \\
&= \frac{(N + 1)^{Kd \delta}}{\frac 2 d \sum_{k = 1}^N \log \frac{1}{R_{k, j_k}} - K \delta \log(N + 1)}.
\end{align*}

However, we can choose the radii $R_{k, j_k}$ as small as we desire, making the right hand side of this equation arbitrarily small. This leads to a contradiction, showing that this uniformly bounded selection of $\eta_{k, j_k}$ was in fact too small. Hence at least some of the selections $\eta_{k, j_k}$ must have been larger than $(1 + 1/k)^{\delta}$, a contradiction. Moreover, asymptotically, the choices of $\eta_{k, j_k}$ must be much larger than $(1 + 1/k)^{\delta}$, and larger choices of $\eta_{k, j_k}$ lead to a larger Cantor set.

%
%

\textbf{Acknowledgements.} I am very grateful to my advisor, Ignacio Uriarte-Tuero, who gave me this problem to work on. His helpful advice and insight was invaluable throughout the work.
%
%

\bibliography{gauged_stretches}{}
\bibliographystyle{plain}
\end{document}